\documentclass[12pt]{article}
\usepackage{pifont}
\usepackage{psfrag,float,subfigure}
\usepackage{float,amsmath,amssymb,amsthm,mathrsfs,graphicx,subfigure,mathrsfs}
\usepackage{caption,enumerate}
\usepackage[numbers,sort&compress]{natbib}
\usepackage[colorlinks]{hyperref}
\usepackage[scale=0.8,a4paper]{geometry}
\usepackage{graphics}
\usepackage{epsfig}
\usepackage{epstopdf}
\usepackage{psfrag}
\usepackage{subfigure}
\usepackage{graphicx}

\newtheorem{lemma}{Lemma}%[section]
\newtheorem{theorem}[lemma]{Theorem}
\newtheorem{corollary}[lemma]{Corollary}

\newtheorem{conjecture}[lemma]{Conjecture}

\title{The number of perfect matchings in a brick}
\author{{Fuliang Lu\thanks{Email address: flianglu@163.com}, Huali Pan}\\
{\small  School of Mathematics and Statistics,
Minnan Normal University, Zhangzhou 363000, China}}
\date{}\makeatother
\begin{document}
\maketitle

\begin{abstract}
%{\bf Abstract} \
  A 3-connected graph is a {\em brick}  if the graph obtained from it by deleting any two distinct vertices has a perfect matching.
  The importance of bricks stems from the fact that they are building blocks of the matching decomposition procedure of Kotzig, and Lov\'asz and Plummer.

 Lucchesi and Murty conjectured that there exists a positive integer $N$ such that for every $n\geq N$, every brick on $n$ vertices
has at least $n-1$ perfect matchings. We
present an infinite family of bricks such that for each  even integer  $n$ ($n>17$), there exists a brick with $n$ vertices in this family that  contains $\lceil0.625n\rceil$  perfect matchings, showing that this conjecture fails.
\end{abstract}


\section{Introduction}
All graphs in this paper are finite and contains no loops (multiple edges are allowed).
 We follow \cite{BM08} for undefined notation and terminology.
 For a graph $G$, we denote by $V(G)$ and $E(G)$ the vertex set and edge set of $G$, respectively.
A connected nontrivial graph  is \emph{matching covered} if each of its edges lies in some perfect matching. Matching covered graphs are also called 1-extendable graphs \cite{LP86}.

 The problem of computing the number of perfect matchings in a graph has been much studied in combinatorics and
has connections to problems in molecular chemistry, statistical physics; see for instance [7, Section 8]. In general graphs, this problem is $\sharp$P-complete \cite{va}. We may focus on matching covered graphs when we consider the problem, as the removal of  edges not in any perfect matchings would not affect the number of perfect matchings.

For $X\subseteq V(G)$, by $\partial_G(X)$ we mean the {\it edge cut} of $G$, which is the set of edges of $G$ with one end in $X$ and the other in  $\overline{X}$, where $\overline{X}=V(G)\backslash X$; by $G/(X\rightarrow x)$ or simply $G/X$ we mean the graph obtained by contracting $X$ to a single vertex $x$, the graph $G/(\overline{X}\rightarrow \overline{x})$ or simply $G/\overline{X}$ is defined analogously. We shall refer to these
two graphs $G/X$ and $G/\overline{X}$ as the $\partial_G(X)$-contractions of $G$.
Let $G$ be a matching covered graph.
An edge cut $C=\partial_G(X)$ of  $G$ is {\it tight} if $|M\cap C|=1$ for each perfect matching $M$ of $G$ and is {\it separating} if $G/X$ and $G/{\overline X}$ are matching covered. A tight cut $\partial_G(X)$ is {\it trivial} if either $|X|=1$ or $|\overline{X}|=1$, and {\it nontrivial} otherwise.
A matching covered graph that is free of nontrivial tight cuts is a {\it brace} if
it is bipartite, and a \emph{brick} if it is nonbipartite.
The importance of bricks and braces stems from the fact that they are building blocks of matching covered graphs by the matching
decomposition procedure of Kotzig, and Lov\'asz and Plummer \cite{LP86}. Furthermore,
Lov\'asz~\cite{Lovasz87} proved
that any matching covered graph can be decomposed
into a unique list of bricks and braces  by the tight cut decomposition.
 %In particular, any two decompositions of a matching covered graph $G$ yield the same number of bricks; this number is denoted by $b(G)$.

 A classical theorem
of Petersen states that every 2-connected cubic graph has at least one perfect matching \cite{pe}.  It can be proven that every  2-connected cubic graph is  matching covered \cite{ple},  which implies that it contains at least 3 perfect matchings.
Confirming a conjecture of Lov\'asz and Plummer, Esperet,
Kardo\v s, King, Kr\'al' and Norine showed that every 2-connected cubic
graph on $n$ vertices has at least $2^{\frac{n}{3656}}$ perfect matchings \cite{es}.
 Carvalho,   Lucchesi and  Murty showed
that any brace on $n$ vertices has at least $\frac{(n-2)^2}{8}$ perfect matchings \cite{CLM13}.
  Carvalho,   Lucchesi and  Murty \cite{CLM05}
defined a brick $G$ to be {\em extremal}
if the number of perfect matching of $G$ is equal to the dimension of the lattice spanned by the set of incidence vectors of perfect matchings of $G$, that is $|E(G)|-|V(G)|+1$. Clearly the number of perfect matchings in a brick $G$ is at least $|E(G)|-|V(G)|+1$.

For an integer $k\geq3$, the \emph{wheel} $W_k$ is the graph obtained from a cycle $C$ of length $k$ by adding a new vertex $h$ and joining it to all vertices of $C$. The vertex $h$ is its \emph{hub}, and the edges incident
with $h$ are its \emph{spokes}. An odd wheel
is a wheel with an odd number of spokes.  It can be checked that each odd wheel is a brick, and  every spoke of the odd wheel lies in exactly one perfect matching. So the number of perfect matchings in
an odd wheel with  $n$ vertices is $n-1$.
Lucchesi and Murty conjectured that the lower bound  of the number of perfect matchings in a brick is  a linear function  of the number of vertices.
  More exactly, see the following conjecture (unsolved problems No.5 in \cite{Lucchesi2024}).

\begin{conjecture}{\em \cite{Lucchesi2024}}\label{Lucchesi2024}
There exists a positive integer $N$ such that for every $n\geq N$, every brick on $n$ vertices
has at least $n-1$ perfect matchings.
\end{conjecture}

We
construct an infinite family of bricks on $n$ $(n>17)$ vertices that  contains $\lceil0.625n\rceil$ perfect matchings (see Theorem \ref{main}), which shows  that this conjecture fails. This family of bricks
will be presented in Section 3. We will give some useful results in the following section.

\section{Preliminaries}
We begin with some notation.  For $X\subseteq V(G)$,  by $N_G(X)$, or simply $N(X)$, we mean the set of vertices that are not in $X$ but have neighbors in $X$.
An edge of a graph is {\em solitary} if it lies in precisely one perfect matching of the
graph (solitary edges appeared in benzenoid hydrocarons of
theoretical chemistry under name ``forcing edges").  A brick is \emph{solid} if it is free of nontrivial separating cuts.

Let $G$ and $H$ be two vertex-disjoint graphs and let $u$ and $v$ be vertices of $G$ and $H$, respectively, such that $d_G(u)=d_H(v)$.
Moreover, let $\theta$ be a given bijection between $\partial_G(u)$ and $\partial_H(v)$.
We denote by $(G(u)\odot H(v))_\theta$ the graph obtained from the union of $G-u$ and $H-v$ by joining, for each edge $e$ in $\partial_H(v)$, the end of $e$ in $H$ belonging to $V(H)-v$ to the end of $\theta(e)$ in $G$ belonging to $V(G)-u$;
and refer to $(G(u)\odot H(v))_\theta$ as the graph obtained by \emph{splicing $G$ (at $u$), with $H$ (at $v$), with respect to the bijection $\theta$}, for brevity, to $G(u)\odot H(v)$.
In general, the graph resulting from splicing two graphs $G$ and $H$ depends on the choice of $u$, $v$ and $\theta$.
Let $X=V(G)\setminus\{u\}$. Obviously,   $G(u)\odot H(v)/X\cong H$, $G(u)\odot H(v)/\overline{X}\cong G$.

Specially,
if $H=K_4$ and $d_G(u)=3$, then the splicing operation  $G(u)\odot K_4(v)$ can be intuitively viewed as the operation that `inserts a triangle' at $u$. So we also call such an operation the {\it triangle-insertion} at $u$ of $G$ and denote  $G(u)\odot K_4(v)$ simply by $G\langle u\rangle$.
As an extension of triangle-insertion, for vertices $v_1,v_2,\ldots,v_k$ of $V(G)$, we denote the graph obtained from $G$ by the triangle-insertions at  $v_1,v_2,\ldots,v_k$ by $G\langle \{v_1,v_2,\ldots,v_k\}\rangle$.

 Edmonds, Lov\'asz and   Pulleyblank \cite{ELP82} showed that a graph $G$ is a \emph{brick} if and only if $G$ is 3-connected and for any two distinct vertices $x$ and $y$ of $G$, $G-\{x,y\}$ has a perfect matching. Moreover, we have the following lemma.
\begin{lemma}{\em\cite{CLM05}} \label{lem:extremal}
Let $G$ be a matching covered graph, $C=\partial_G (X)$ be a nontrivial
separating cut of $G$ such that each $C$-contraction of $G$ is a brick. Then, $G$ is a
brick if and only if no pair of vertices of $G$, one in $X$, the other in $\overline{X}$, covers the set
of edges of $C$.
\end{lemma}

The following corollary can be gotten by Lemma \ref{lem:extremal} directly.
\begin{corollary}\label{cor:tri}Let $G$
be a brick and $u$ be a vertex of degree 3 in $G$. Then  $G\langle u\rangle$ is a  brick.
\end{corollary}

\begin{theorem}{\em\cite{CLM05}}\label{Th:splicing}
Let $W$ be an odd wheel with hub $h$, up to multiple edges
incident with $h$. Let $G$ be a graph obtained by the splicing of $W$ at $h$ and an
extremal brick. Let $C$ denote the cut $\partial_W(h)=\partial_G(V(W)\setminus\{h\})$. If $G$ has
precisely one perfect matching  that contains more than one edge in $C$, then $G$
is a nonsolid extremal brick.
\end{theorem}

The following corollary can be gotten by Corollary \ref{cor:tri} and Theorem \ref{Th:splicing} directly.
\begin{corollary}\label{cor:triangle} Let $G$
be an extremal brick, and $u$ be a vertex of degree 3 in $G$. If $G-u-N(u)$ has only one prefect matching, then  $G\langle u\rangle$ is an extremal brick.
\end{corollary}
\begin{figure}[htbp]
\begin{center}
\includegraphics[width=4cm]{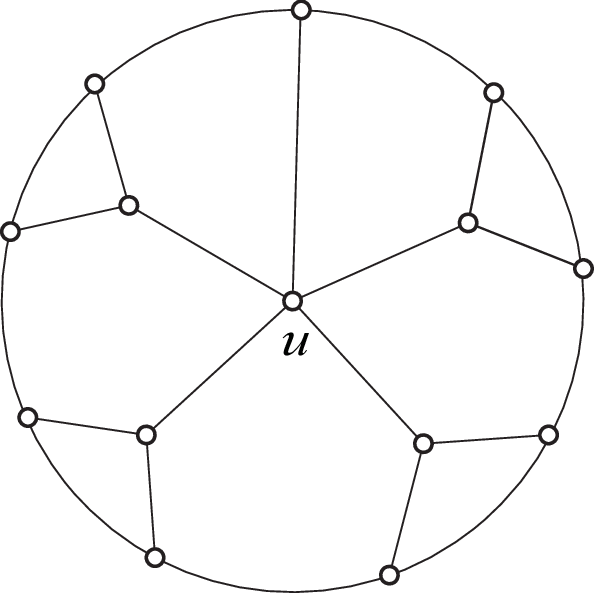}
\caption{ ${G}_0$.}\label{fig:0}\end{center}
\end{figure}

 \begin{lemma}\label{lemma:g0} Let $G_0$ be the graph gotten by the triangle-insertions at four vertices on the cycle of a $W_5$ (see Figure \ref{fig:0}). Then $G_0$ is an extremal brick.
\end{lemma}
  \begin{proof}Let $u$ be the hub of $W_5$. By Corollary \ref{cor:triangle} repeatedly, $G_0$ is a brick. Note that $|V(G_0)|=14$, $|E(G_0)|=22$. Moreover, for every vertex $v$ in $ N_{G_0}(u)$, if  $v$ lies in a triangle of $G_0$, then $G_0-u-v$ contains two perfect matchings; otherwise, $G_0-u-v$ contains one perfect matching. So, the brick $G_0$ has 9 perfect matchings. By the definition,  $G_0$ is extremal.
\end{proof}

\section{The main theorem}

 Let $u_0^1u_0^2u_0^3u_0^4$, $u_1^1u_1^2u_1^3u_1^4,\dots, u_t^1u_t^2u_t^3u_t^4$ be $t+1$ paths with 4 vertices, where $t>0$. Assume that  ${G}_t$ is the  graph  gotten from the union of the $t+1$ paths by adding a new vertex $u$ and edges $\{u_s^1u_{s+1}^1, u_s^4u_{s+1}^3:s=0,1,\ldots,t-1\}\cup \{uu_t^1,uu_t^4\}$  to the $t+1$ paths. Let $S=\{u_0^1,u_0^3, u\}\cup \{ u_s^2, u_s^4:s=0,1,\ldots,t\}$. Then $|S|=2t+5$.
Let $S_1=S\setminus\{u_t^2, u_t^4,u\}$,  $S_2=S\setminus\{ u_t^4,u\}$,
$S_3=S\setminus\{u\}$, $S_4=S$.
Let  ${G}'_t$ be the  graph  gotten from $G_t$ by adding a new vertex $u'$ and edges $\{u'v:v\in S\}$. See the left of Figures \ref{fig:1} and \ref{fig:2} for example (when $t=1$ and $t=3$).
Obviously, ${G}_t$ contains odd number of vertices, while $|V({G}'_t)|$ is even.
And in ${G}_t$, every vertex in $S$ is of degree 2,  the other vertices are of degree 3.

Let $n$ be an even integer such that  $n=i+8(t+1)>17$, where $0< i\leq 8$.
 Let  ${G}''_n= {G}'_t \langle S_{\frac{i}{2}}\rangle$.
 Then $|V({G}''_n)|=n$. Label the triangle  inserted at vertex $u$ as $x,y,z$ such that $\{xu', yu_t^1\}\subset E(G_n'')$; label the triangle  inserted at vertex $u_s^j$ as $x_s^j,y_s^j,z_s^j$ such that $\{x_s^ju', y_s^ju_s^{j-1}\}\subset E(G_n'')$ for $s\geq1$ and $j=2,4$;
the labels  of other triangles see the right of Figures  \ref{fig:1} or \ref{fig:2}. Let $ X_t=\{u^1_t,u^2_t, u^3_t,u^4_t,u \}$,
 $\mathcal{G'}=\{G_t':t=1,2,3,\ldots\}$, $\mathcal{G''}=\{G_n'':n=18,20,22,\ldots\}$. Let $M_0=\{x^1_{0}y^1_{0},x^2_{0}z^2_{0},z^1_{0}y^2_{0},x^3_{0}y^3_{0},x^4_{0}z^4_{0},z^3_{0}y^4_{0}\}$, and $M_t=M_0\cup(\cup_{s=1}^{t} \{x_s^4z_s^4,y_s^4u_s^3,x_s^2z_s^2,y_s^2u_s^1\})$.
  The following is our main theorem.

\begin{theorem}\label{main}Every graph in $\mathcal{G''}$ is an extremal brick. Therefore,
the number of perfect matchings of  $G_n''$  is $(5n+8-i)/8$, where $n=i+8(t+1)\geq 18$ and $0< i\leq 8$.
\end{theorem}

\begin{proof}Firstly, we have the following claims.

{\bf Claim 1.}\label{lemma:brick}
Every graph  in $\mathcal{G'}$ is a brick.
\begin{proof} Note that $G_t'/(\overline{X_t}\rightarrow \overline{x})$ is a $W_5$ with $\overline{x}$ as its hub. We will show the result by induction on $t$.
 If $t=1$, then $G_1'/{X_1}$ is $W_5$. Note that $\{u_0^1u_1^1, u_0^4u_1^3, u_1^2u'\}$ is a matching in $\partial_{G_1'} (X_1)$. As the odd wheel $W_5$ is a brick,
$G_1'$ is a brick  by Lemma \ref{lem:extremal}. So we assume that the result holds for $t<k$. Now we consider the case when $t=k$.
Then $G_k'/{X_k}$ is isomorphic to the graph gotten from $G_{k-1}'$ by adding two multiple edges between $u$ and $u'$. By induction hypothesis, $G_{k-1}'$ is a brick. By the definition of the brick, adding a multiple edge to a brick is still   a brick.
Note that $\{u_{k-1}^1u_k^1, u_{k-1}^4u_k^3, uu'\}$ is a matching in $\partial_{G_k'} (X_k)$.
By Lemma \ref{lem:extremal} again, the claim follows.
\end{proof}
{\bf Claim 2.} \label{lemma:brick}
Every graph  in $\mathcal{G''}$ is a brick.
\begin{proof}
Note that each graph  in $\mathcal{G''}$ can be gotten from some graph in $\mathcal{G'}$ by the triangle-insertions at several vertices of degree 3. Every graph  in $\mathcal{G'}$ is a brick  by Claim 1. So the claim follows by Corollary \ref{cor:tri} repeatedly.
\end{proof}

{\bf Claim 3.} If $n\equiv 6\mod 8$, then $uu'$ is a solitary edge of ${G''_n}$.
\begin{proof}  In
 $ G_n''-\{u,u'\}$,  the edge $y_t^4u_t^3$ is a cut edge that connects the triangle $x_t^4y_t^4z_t^4$ to the rest part of the graph. So $x_t^4z_t^4 $ and $y_t^4u_t^3 $ lie in every perfect matchings of  $ G_n''-\{u,u'\}$ (if has). Also, the edge $y_t^2u_t^1$ is a cut edge that connects the triangle $x_t^2y_t^2z_t^2$ to the rest part of the graph $ G_n''-\{u,u',x_t^4, y_t^4,z_t^4,u_t^3 \}$.
So $x_t^2z_t^2$ and $y_t^2u_t^1 $ lie in the perfect matchings of $G_n''-\{u,u',x_t^4, y_t^4,z_t^4,u_t^3 \}$. Similarly, the edge set $ \cup_{s=1}^{t-1} \{x_s^4z_s^4,y_s^4u_s^3,x_s^2z_s^2,y_s^2u_s^1\} $ lies in every perfect matching of $ G_n''-\{u,u'\} $. With the same reason, $M_0$ lies the perfect matching of $G_n''-\{u,u' \}- \cup_{s=1}^{t} \{x_s^4,z_s^4,y_s^4,u_s^3,x_s^2,z_s^2,y_s^2,u_s^1\}$.
 Then $M_t$ is the only perfect matching of  $ G_n''-\{u,u'\}$. The claim follows.
\end{proof}

\begin{figure}
\begin{center}
\begin{minipage}{6cm}
\includegraphics[width=4.2cm]{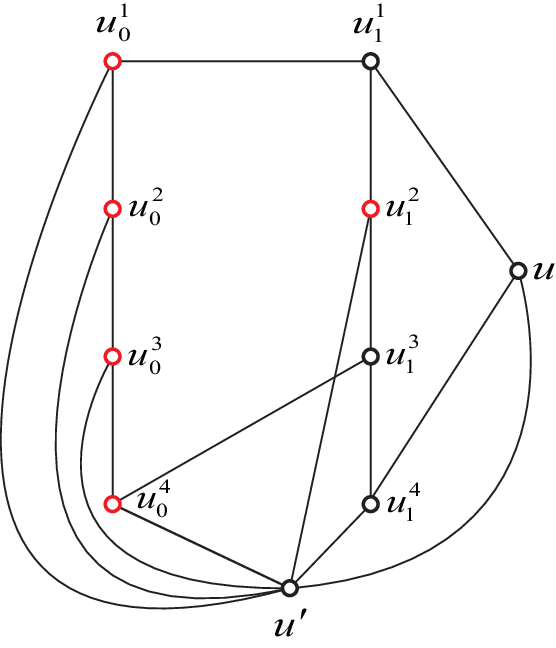}
\end{minipage}
\begin{minipage}{6cm}
\includegraphics[width=5.5cm]{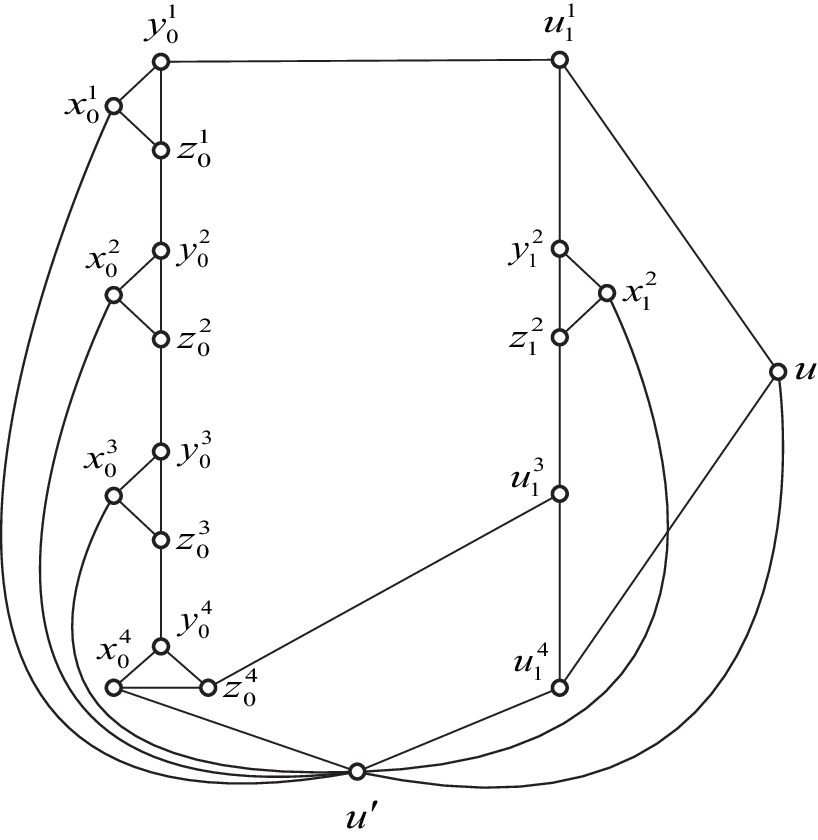}
\end{minipage}
\end{center}
\caption{ ${G}'_1$ and ${G}''_{20}$.}\label{fig:1}
\end{figure}

{\bf Claim 4.} Every graph  in $\mathcal{G''}$ is extremal.
\begin{proof} We will prove the result by induction on $t$.  Firstly, we assume that $t=1$. Then $n\in\{ 18,20,22,24\}$.
If $n\equiv 2\mod 8$, then $n=18$ and $G_n''/(\overline{X_1}\rightarrow \overline{x})$ is a $W_5$ with $\overline{x}$ as its hub. Moreover,
  $G_{18}''/{X_1}$ is isomorphic to $G_0$, which is an extremal brick by Lemma \ref{lemma:g0}.
Note that every edge in $\partial(X_1)$ is incident with one of vertices in $\{y_0^1,z_0^4,u'\}$, and $y_0^1$ and $z_0^4$ are incident with exactly one edge in $\partial(X_1)$ respectively. So the perfect matching of $G_{18}''$ containing more than one edges in $\partial(X_1)$ should contain $y_0^1u_1^1$ and $z_0^4u_1^3$.  As $u_1^2$ has only one neighbor in $G_{18}''-\{y_0^1,u_1^1, z_0^4,u_1^3\}$: $u'$, we have $u_1^2u'$ lies in the perfect matching of $G_{18}''$ containing more than one edges in $\partial(X_1)$. Therefore, the intersection of $\partial(X_1)$ and the perfect matching of $G_{18}''$ containing more than one edges in $\partial(X_1)$ is the edge set $\{y_0^1u_1^1,z_0^4u_1^3,u_1^2u'\}$. It can be checked that
 $G_{18}''-\{y_0^1,u_1^1,z_0^4,u_1^3,u_1^2,u'\}$ contains only one perfect matching: $\{uu_1^4,x_0^4y_0^4,x_0^3z_0^3, z_0^2y_0^3,  x_0^2y_0^2,x_0^1z_0^1\}$.
  By Theorem \ref{Th:splicing}, $G_{18}''$ is extremal.

If $n=20$, then $G_{20}''\cong G_{18}'' \langle u_1^2\rangle$.
It can be checked $M_0\cup \{uu_1^4\}$ is the only perfect matching of $G_{18}''-\{u_1^2 \}-N( u_1^2)$. As the degree of $u_1^2$ is three in $G_{20}''$, the result holds in this case by Corollary \ref{cor:triangle}.
 Similarly,
if $n=22$, then $G_{22}''\cong G_{20}'' \langle u_1^4\rangle$, and $G_{20}''-\{u_1^4 \}-N( u_1^4)$ has only one perfect matching: $M_0\cup \{x_1^2z_1^2,y_1^2u_1^1\}$;
if $n=24$, then $G_{24}''\cong G_{22}'' \langle u\rangle$, and $G_{22}''-\{u\}-N(u)$ has only one perfect matching: $M_0\cup \{x_1^2y_1^2,u_1^3z_1^2,x_1^4y_1^4\}$. By Corollary \ref{cor:triangle} again, the result follows when $t=1$.

\begin{figure}
\begin{center}
\begin{minipage}{7cm}
\includegraphics[width=6cm]{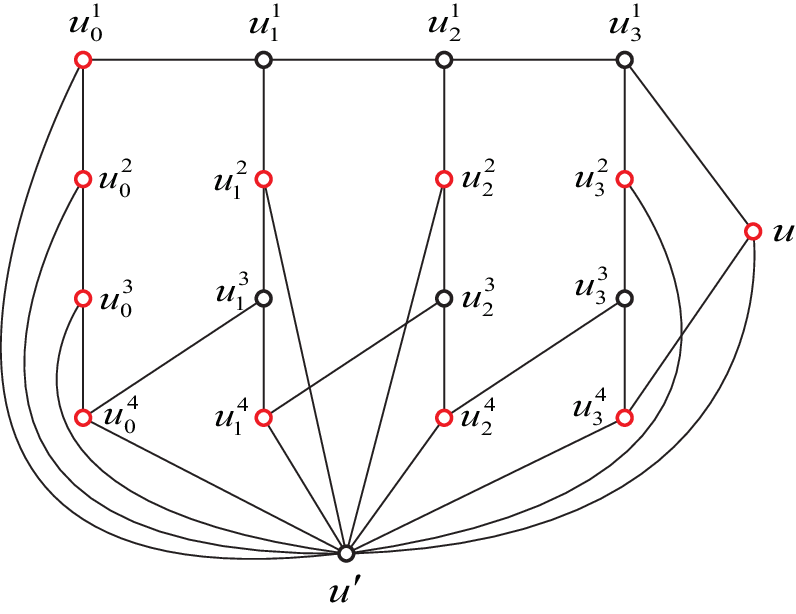}
\end{minipage}
\begin{minipage}{7cm}
\includegraphics[width=8cm]{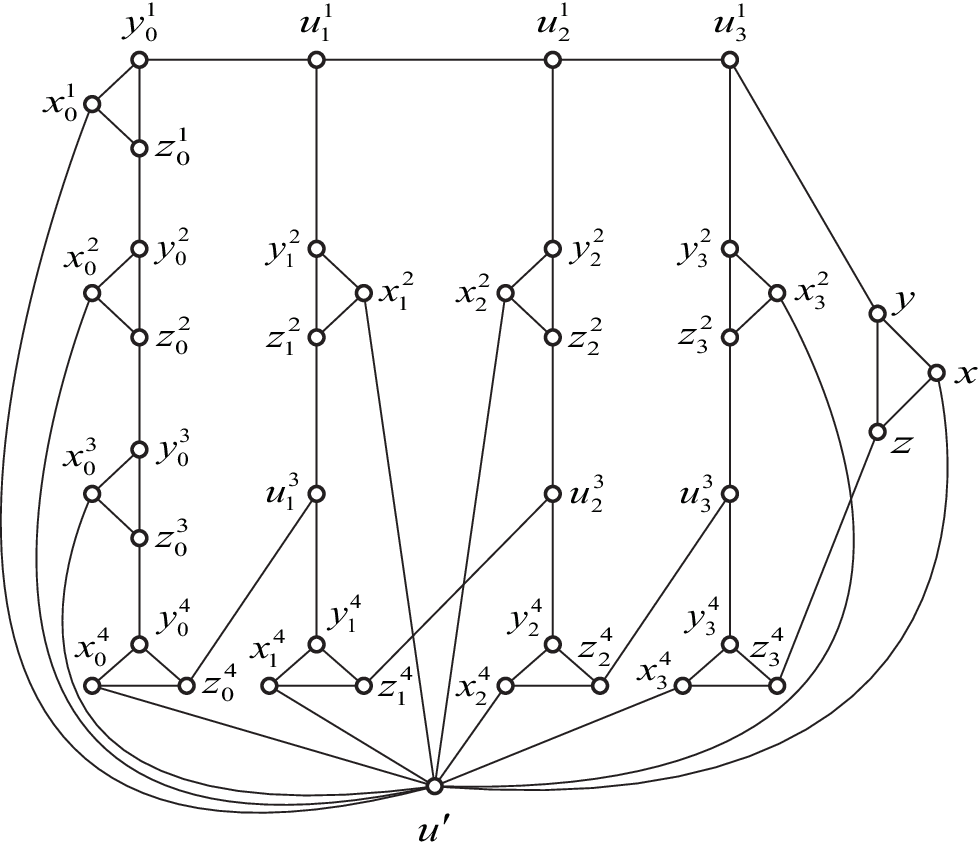}
\end{minipage}
\end{center}
\caption{ ${G}'_3$ and ${G}''_{40}$.}\label{fig:2}
\end{figure}

Now we assume that the result holds for $1<t<k$. Following, we consider the case when $t=k$.
If $n\equiv 2\mod 8$,
then $G_n''/{X_k}$ is isomorphic to the graph gotten from $G_{n-4}''$ by adding two multiple edges between $u$ and $u'$.
As $n\equiv 2\mod 8$, we have $n-4\equiv 6 \mod 8$, that is $n-4=6+8k$.
By Claim 3,  $uu'$ is a solitary edge in $G_{n-4}''$. So
 $G_n''/{X_k}$ contains two more perfect matchings than $G_{n-4}''$. As $G_{n-4}''$ is extremal by induction hypothesis,  $G_n''/{X_k}$ is extremal by the definition.
 Similar to the case when $n=18$, the intersection of $\partial(X_k)$ and the perfect matching of $G_{n}''$ containing more than one edges in $\partial(X_k)$ is  the edge set $\{u_{k-1}^1u_k^1,z_{k-1}^4u_k^3,u_k^2u'\}$.
 In  $G_{n}''- \{u_{k-1}^1,u_k^1,z_{k-1}^4,u_k^3,u_k^2,u'\} $, $x_{k-1}^4$ and $u$ is of degree 1,
 So, $x_{k-1}^4y_{k-1}^4$ and $uu_{k}^4$ lie in every perfect matching of $G_{n}''- \{u_{k-1}^1,u_k^1,z_{k-1}^4,u_k^3,u_k^2,u'\} $. With the same reason, the edge set
 $\{  u_{k-1}^3z_{k-1}^2,  x_{k-1}^2y_{k-1}^2 \}$ lies in every perfect matching of $G_{n}''- (\{u_{k-1}^1,z_{k-1}^4,u', x_{k-1}^4,y_{k-1}^4,u\}\cup(\cup_{i=1}^4\{u_k^i\}))$.
  Note that
 $G_{n}''- (\cup_{i=1}^4\{u_k^i\}\cup(\cup_{i=2,4}\{z_{k-1}^i,  x_{k-1}^i,y_{k-1}^i \})\cup
 \{u_{k-1}^1,u', u, u_{k-1}^3 \}) $
 is isomorphic to $G_{n-12}'' -\{u,u'\}$ when $n>26$ (if $n=26$,  $G_{n}''- (\cup_{i=1}^4\{u_k^i\}\cup(\cup_{i=2,4}\{z_{k-1}^i,  x_{k-1}^i,y_{k-1}^i \})\cup
 \{u_{k-1}^1,u', u, u_{k-1}^3 \} )$ is isomorphic to the graph gotten from  a path with 12 vertices  by adding 4 edges forming 4  vertex-disjoint triangles, which has exactly one perfect matching).
  As $n-12\equiv 6 \mod 8$, $G_{n-12}'' -\{u,u'\}$ has exactly one perfect matching by Claim 3.
So,   $G_{n}''- \{u_{k-1}^1,u_k^1,z_{k-1}^4,u_k^3,u_k^2,u'\} $ has exactly one perfect matching. Therefore,
 the result follows by Theorem \ref{Th:splicing} when $n\equiv 2\mod 8$.

If $n\equiv 4\mod 8$, then $ G_n''\cong G_{n-2}''\langle u_k^2\rangle$. Note that $N_{G_{n-2}''}(u_k^2)=\{u_k^1,u_k^3,u'\}$. In
$ G_{n-2}''- u_k^2-N(u_k^2)$, $u$ is of degree 1. So $uu_k^4$  belongs to every perfect matching of $G_{n-2}''- u_k^2-N(u_k^2)$.  Note that $G_{n-2}''-\{u,u_k^4,u_k^2\}-N(u_k^2)$ is isomorphic to
   $G_{n-6}''- \{u,u'\}$, and $n-6\equiv 6\mod 8$. By Claim 3, $G_{n-6}''- \{u,u'\}$ contains only  one perfect matching. So $ G_{n-2}''- u_k^2-N(u_k^2)$ has exactly one perfect matching.
As $n-2\equiv 2\mod 8$, $ G_{n-2}''$ is extremal by last paragraph. So
 the result follows by Corollary \ref{cor:triangle} when $n\equiv 4\mod 8$.

If $n\equiv 6\mod 8$, then $ G_n''\cong G_{n-2}''\langle u_k^4\rangle$.
Note that $N_{G_{n-2}''}(u_k^4)=\{u,u',u_k^3\}$, and the edge set $\{x_k^2z_k^2, u_k^1y_k^2\}\cup M_{k-1}$ is the only perfect matching of $ G_{n-2}''- u_k^4-N(u_k^4)$.
Similar to last paragraph,  the result follows by Corollary \ref{cor:triangle} in this case.
If $n\equiv 0\mod 8$, then $ G_n'' \cong G_{n-2}''\langle u\rangle$. It can be checked that $\{x_k^2y_k^2,x_k^4y_k^4,z_k^2u_k^3 \}\cup M_{k-1}$ is the only perfect matching of $ G_{n-2}''- u-N(u)$. Note that $|N_{G_{n-2}''}(u)|=3$. Similarly, the result follows by Corollary \ref{cor:triangle} again.
\end{proof}

Note that
 for every graph in $G_n''$, except one vertex is of degree $2t+5$, every vertex is of degree 3. So $|E(G_n'')|=(2t+5+3(n-1))/2=t+1+3n/2$. Therefore, $|E(G_n'')|-|V(G_n'')|+1=t+2+n/2 =\lceil\frac{5n}{8}\rceil $ (recall that $n=i+8(t+1)$ and $0< i\leq 8$). By Claim 4, $G_n''$ is extremal. So the result follows.\end{proof} %$\hfill\square$

It is known that every wheel is solid. Carvalho,   Lucchesi and   Murty \cite{CLM05} showed that for  a solid brick $G$ on $n$ vertices and $n>4$,   $G$
is extremal and if and only if it is an odd wheel up to multiple spokes.
 Note that every graph in $\mathcal{G''}$ is not solid. It may be true that the Conjecture \ref{Lucchesi2024} holds for solid bricks.

\end{document}